\documentclass[12pt,reqno]{amsart}
\usepackage{amssymb, amsmath,amsthm}
\newtheorem{thm}{Theorem}

\newtheorem{Lemma}{Lemma}
\newtheorem{Corollary}{Corollary}

\newtheorem{Example}{Example}
\numberwithin{defn}{section}
\numberwithin{thm}{section}
\numberwithin{Lemma}{section}
\numberwithin{Corollary}{section}
\numberwithin{Example}{section}
\numberwithin{subsection}{section}
\numberwithin{Remark}{section}
\numberwithin{equation}{section}
\numberwithin{ppn}{section}
\usepackage{xcolor}
\begin{document}
\title
[Convergence Criteria of a Three Step Scheme...]
{Convergence Criteria of a Three Step Scheme under generalized Lipschitz Condition in Banach Spaces} 
\author
[Saxena et al.]
{Akanksha Saxena$^1$, J. P. Jaiswal$^{2}$, K. R. Pardasani$^1$}
\thanks{$^1$ Department of Mathematics, Maulana Azad National Institute of Technology, Bhopal, M.P., India-462003, Email:  AS(akanksha.sai121@gmail.com); KRP(kamalrajp@rediffmail.com)}
\thanks{$^{2}$ Department of Mathematics, Guru Ghasidas Vishwavidyalaya (A Central University), Bilaspur, C.G., India-495009, Email: asstprofjpmanit@gmail.com
}
\date{}
\maketitle

\textbf{Abstract.} 
The goal of this study is to investigate the local convergence of a three-step Newton-Traub technique for solving nonlinear equations in Banach spaces with a convergence rate of five. The first order derivative of a nonlinear operator is assumed to satisfy the generalized Lipschitz condition, i.e. the $\kappa$-average condition.
Furthermore, a few results on the convergence of the same method in Banach spaces are developed under the assumption that the derivative of the operators satisfies the radius or center Lipschitz condition with a weak $\kappa$-average, and that $\kappa$ is a positive integrable function but not necessarily non-decreasing.
Our new notion provides a tighter convergence analysis without the need for new circumstances. As a result, we broaden the applicability of iterative approaches. Theoretical results are supported further by illuminating examples.
The existence and uniqueness of the solution $x^*$ are examined in the convergence theorem.
In the end, we achieve weaker sufficient convergence criteria and more specific information on the position of the solution than previous efforts requiring the same computational effort. We obtain the convergence theorems as well as some novel results by applying the results to some specific functions for $\kappa(u)$. A numerical test is carried out to corroborate the hypothesis established in this work.
\\\\
\textbf{Mathematics Subject Classification (MSC 2020).} 
65H10, 65J15, 65G99, 47J25.
\\\\
\textbf{Keywords and Phrases.} 
Nonlinear problem, convergence radius, local convergence, Banach space, Lipschitz condition, $\kappa$-average.

\section{\bf Introduction}

Let a nonlinear operator $G:\Omega\subseteq X\to Y$ be such that $X$ and $Y$ are two Banach spaces, $\Omega$ is a non-empty open convex subset and $G$ is Fr\'echet differentiable nonlinear operator. 
A popular iterative method for solving the equation
\begin{eqnarray}\label{eqn:11}
G(x)=0,
\end{eqnarray}
is Newton's scheme, which may be represented as:
\begin{eqnarray}\label{eqn:13}
x_{t+1}&=&x_t-[G'(x_t)]^{-1}G(x_t), \ t\geq0,
\end{eqnarray}
is being preferred though its speed of convergence is low. Newton's method {\cite{Ortega}}, is a well known iterative approach that converges quadratically, which was initially studied by Kantorovich \cite{Kantorovich} and then re-evaluated by Rall \cite{Rall}. 

Some Newton-type methods with third-order, fourth-order convergence that do not require the computation of second order derivatives have been developed in the refs.[\cite{Homeier}, \cite{Li}, {\cite{Kang}}, \cite{Traub}, {\cite{Zhanlav}].  While methods of higher $R$-order convergence are often not conducted regularly despite having a high speed of convergence, this is due to the high operational expense. However, the method of higher $R$-order convergence can be applied in stiff system issues \cite{Kantorovich} where quick convergence is necessary.

 The convergence domain is critical for the steady behaviour of an iterative method from a numerical stand point. There are two types of iterative method convergence research: semilocal and local convergence analysis. The semilocal convergence study uses information around a starting point to provide criteria for ensuring the convergence of iterative methods, whereas the local one uses information around a solution to estimate the radii of the convergence balls. Numerous authors investigated the local convergence analysis for Newton-type, Jarratt-type, Weerakoon-type, and other types in Banach space in works [{\cite{Cho}}, \cite{Chen}, \cite{Kanwar}, \cite{Magrenan}, \cite{Sharma}].

Here, we analyze the local convergence of the classical fifth-order Newton-Traub method \cite {Argyros} under the $\kappa$-average condition which is written as:
\begin{eqnarray}\label{eqn:12}
y_t&=&x_t-[G'(x_t)]^{-1}G(x_t),\nonumber\\
z_{t}&=&y_t-[G'(x_t)]^{-1}G(y_t),\nonumber\\
x_{t+1}&=&z_t-[G'(y_t)]^{-1}G(z_t), \ t=0,1, 2,\cdots.
\end{eqnarray}
The important feature of the method $(\ref{eqn:12})$  is that: it is the easiest and most efficient fifth-order iterative method, requiring only two evaluations of the function $G_j$, per j$^{th}$ iteration, one of the first derivative $G_j'$ and no evaluations of the second derivative $G_j''$ making it mathematically effective. Several research on the weakness and/or expansion of hypotheses made on the underlying operators can be found in the literature.
Wang \cite{Wang} created generalized Lipschitz conditions to investigate the local convergence of Newton-Traub method, in which a non-decreasing positive integrable function was chosen rather than the normal Lipschitz constant. In the coming years, Wang and Li \cite{Wang2} has obtained some results on Newton-Traub method convergence in Banach spaces when the derivative of the operators satisfies the radius or center Lipschitz condition but with a weak $\kappa$-average. Shakhno {\cite{Shakhno}} has explored the local convergence of the two-step Secant-type approach \cite{Kantorovich}, where the first-order divided differences satisfy the generalized Lipschitz criteria. Recently, the local convergence of a two-step Newton type method of convergence rate three under generalized Lipschitz conditions has been studied by Saxena et. al \cite{Saxena} whose definitions will be used in this article.

%
As a motivational example let $X=Y=R^3, D=\overline{V}(0,1)$ and $X^*=(0,0,0)^T$. Define function $t$ on $D$ for $w=(x,y,z)^T$ by
 \begin{eqnarray*}
t(w)=(e^x-1, \frac{e-1}{2}y^2+y,z)^T.
\end{eqnarray*}
Then, the Fr\'echet derivative is 
\begin{equation}
t'(w)=
\begin{pmatrix}
e^x && 0 && 0\\
0 && (e-1)y+1 && 0\\
0 && 0 && 1
\end{pmatrix}.
\end{equation}
Hence, $\kappa=\frac{e}{2}, \kappa_0=\frac{e-1}{2}$ and $\kappa_0<\kappa$ (see definitions $(\ref{eqn:21})$ and $(\ref{eqn:22})$). So, replacing $\kappa$ by $\kappa_0$ at the denominator gives the benefits. If $\kappa$ and $\kappa_0$ are not constants then we can take $\kappa(u)=\frac{eu}{2}, \kappa_0(u)=\frac{(e-1)u}{2}\ and\ \overline{\kappa}(u)=\frac{e^{\frac{1}{(e-1)}}u}{2} $ (see definitions $(\ref{eqn:23})$, $(\ref{eqn:24})$ and $(\ref{eqn:556})$).
%

The fascinating question then arises as to whether the radius Lipschitz condition with $\kappa$-average and non-decreasing of $\kappa$ are required for the convergence of fifth-order Newton-Traub method.  In this paper, we derived certain theorems for scheme $(\ref{eqn:12})$, motivated and influenced by the above-mentioned research efforts in this direction. Throughout the first result, generalized Lipschitz conditions were utilized to explore local convergence, which is significant for enlarging the convergence area without the need for additional assumptions, as well as an error estimate. The domain of uniqueness of solution has been derived under the center Lipschitz condition in the second theorem. A few corollaries are also mentioned.


The remainder of this paper is organized as follows: The definitions for $\kappa$-average conditions are found in section $2$.  Sections $3$ and $4$ discuss local convergence and its region of uniqueness, accordingly. The assumption that the derivative of $t$ satisfies the radius and center Lipschitz condition with weak $\kappa$-average, namely $\kappa$ and $\kappa_0$, is improved in Section $5$. $\kappa$ and $\kappa_0$ are claimed to belong to some family of positive integrable functions that are not necessarily non-decreasing for convergence theorems. To demonstrate the significance of the findings, numerical examples are provided.



\section{\bf Special and Generalized Lipschitz conditions }
Throughout this context, $B(x^*, \delta)=\{x:||x-x^*||<\delta\}$ denotes a ball with radius $\delta$ and center $x^*$. \\
{\bf Definition 1.} The constraint placed on the function $G$
\begin{equation}\label{eqn:21}
||G'(x)-G'(y^\tau)||\leq \kappa(1-\tau)(||x-x^*||+||y-x^*||); \forall \ x,\ y \in B(x^*, \delta),
\end{equation}
where $y^\tau=x^*+\tau(y-x^*), 0\leq\tau\leq1,$ and $x=y\ if\ \tau=1$, is usually called radius Lipschitz condition in the ball  $B(x^*, \delta)$ 
with constant $\kappa$. \\ 
{\bf Definition 2.} The constraint imposed on the function $G$
\begin{eqnarray}\label{eqn:22}
||G'(x)-G'(x^*)||\leq 2\kappa_0||x-x^*||; \forall \ x \in B(x^*, \delta),
\end{eqnarray}
we call it the center Lipschitz condition in the ball $B(x^*, \delta)$ with constant $\kappa_0$ where $\kappa_0\leq \kappa$. 

In this scenario, replacing $\kappa$ by $\kappa_0$ in case $\kappa_0<\kappa$ leads to a wider range of initial guesses (bigger radius of convergence than in prior studies) and fewer iterations to attain error tolerance, and the uniqueness of the solution $x^*$ is also expanded ({\cite{Cho}}, \cite{Magrenan}). When $\kappa$ and $\kappa_0$ are not constants but can be a positive integrable function, in this situation, the criteria $(2.1)-(2.2)$ are respectively, are substituted by\\
{\bf Definition 3.}
\begin{eqnarray}\label{eqn:23}
||G'(x)-G'(y^\tau)||\leq \int_{\tau(\rho(x)+\rho(y))}^{\rho(x)+\rho(y)} \kappa(u) du; \forall \ x,\ y \in B(x^*, \delta), 0\leq\tau \leq1,\nonumber\\
\end{eqnarray}
{\bf Definition 4.}
\begin{eqnarray}\label{eqn:24}
||G'(x)-G'(x^*)||\leq \int_{0}^{2\rho(x)} \kappa_0(u) du; \forall \ x \in B(x^*, \delta), 
\end{eqnarray}
where $\rho(x)=||x-x^*||$ and we have $\kappa_0(u)\leq \kappa(u)$. Simultaneously, the equivalent' Lipschitz conditions' are referred to as having the $\kappa$-average or generalized Lipschitz conditions. Following that, we will begin with the below lemmas, that will be used eventually in the fundamental theorems.
\begin{Lemma}\label{lm:21}
Suppose that $G$ has a continuous derivative in $B(x^*, \delta)$ and $[G'(x^*)]^{-1}$ exists.\\
(i) If $[G'(x^*)]^{-1}G'$ satisfies the radius Lipschitz condition with the $\kappa$-average:
{\small
\begin{eqnarray}\label{eqn:25}
||[G'(x^*)]^{-1}(G'(x)-G'(y^\tau))||\leq \int_{\tau(\rho(x)+\rho(y))}^{\rho(x)+\rho(y)} \kappa(u) du; \forall \ x,\ y \in B(x^*, \delta), 0\leq\tau \leq1, \nonumber\\
\end{eqnarray}
}
where $y^\tau=x^*+\tau(y-x^*)$, $\rho(x)=||x-x^*||$ and $\kappa$ is non-decreasing, then we have
\begin{equation}\label{eqn:26}
\int_{0}^{1}||[G'(x^*)]^{-1}(G'(x)-G'(y^\tau))||\rho(y)d\tau\leq \int_{0}^{\rho(x)+\rho(y)}\kappa(u)\frac{u}{\rho(x)+\rho(y)}\rho(y)du.
\end{equation}
(ii) If $[G'(x^*)]^{-1}G'$ satisfies the center Lipschitz condition with the $\kappa_0$-average:
\begin{equation}\label{eqn:27}
||[G'(x^*)]^{-1}(G'(x^\tau)-G'(x^*))||\leq \int_{0}^{2\tau\rho(x)} \kappa_0(u) du; \forall \ x,\ y \in B(x^*, \delta), 0\leq\tau \leq1,
\end{equation}
where $\rho(x)=||x-x^*||$ and $\kappa_0$ is non-decreasing, then we have
\begin{equation}\label{eqn:28}
\int_{0}^{1}||[G'(x^*)]^{-1}(G'(x^\tau)-G'(x^*))||\rho(x)d\tau\leq \int_{0}^{2\rho(x)}\kappa_0(u)\left(\rho(x)-\frac{u}{2}\right)du.
\end{equation}
\end{Lemma}
\begin{proof}
The Lipschitz conditions $(\ref{eqn:25})$ and $(\ref{eqn:27})$, respectively, imply that
{\small
\begin{eqnarray}\label{eqn:29}
\int_{0}^{1}||[G'(x^*)]^{-1}(G'(x)-G'(y^\tau))||\rho(y)d\tau&\leq& \int_{0}^{1}\int_{\tau(\rho(x)+\rho(y))}^{\rho(x)+\rho(y)}\kappa(u) du \rho(y)d\tau\nonumber\\
&=&\int_{0}^{\rho(x)+\rho(y)}\kappa(u)\frac{u}{\rho(x)+\rho(y)}\rho(y)du,\nonumber\\
\int_{0}^{1}||[G'(x^*)]^{-1}(G'(x^\tau)-G'(x^*))||\rho(x)d\tau&\leq& \int_{0}^{1}\int_{0}^{2\tau\rho(x)} \kappa_0(u) du\rho(x)d\tau\nonumber\\
&=&\int_{0}^{2\rho(x)}\kappa_0(u)\left(\rho(x)-\frac{u}{2}\right)du.\nonumber
\end{eqnarray}}
where $x^\tau=x^*+\tau(x-x^*)$ and $y^\tau=x^*+\tau(y-x^*)$.
\end{proof}

\begin{Lemma}\cite{Wang2}\label{lm:22}
Suppose that $\kappa$ is positive integrable. Assume that the function $\kappa_a$ defined by relation $(\ref{eqn:56})$ is non-decreasing for some $a$ with $0\le a\le1$.
Then, $for\ each \ \beta\geq0$, the function $\varphi_{\beta,a}$ defined by
\begin{equation}\label{eqn:29}
\varphi_{\beta,a}(P)=\frac{1}{P^{a+\beta}}\int_{0}^{P}u^\beta\kappa(u)du,
\end{equation}
is also non-decreasing.
\end{Lemma}

\section{\bf Local convergence of  Newton-Traub  method $(\ref{eqn:12})$}
Throughout this section, we prove the existence theorem for the Newton-Traub method $(\ref{eqn:12})$ under the radius Lipschitz condition.

\begin{thm}\label{th:31}
Suppose that $G(x^*)=0$, $G$ has a continuous derivative in $B(x^*, \delta)$, $[G'(x^*)]^{-1}$ exists and $[G'(x^*)]^{-1}G'$ satisfies $(\ref{eqn:23})$, $(\ref{eqn:24})$, $\kappa_0$ and $\kappa$ are non-decreasing. Let $\delta$ satisfy the relation
\begin{eqnarray}\label{eqn:32}
\frac{ \int_{0}^{2\delta} \kappa(u)u du}{2\delta(1-\int_{0}^{2\delta} \kappa_0(u) du)}\leq 1.
\end{eqnarray}
Then, the three-step Newton-Traub method $(\ref{eqn:12})$ is convergent for all  \ $x_0\in B(x^*, \delta)$ and 
\begin{eqnarray}\label{eqn:34}
||y_t-x^*||\leq\frac{ \int_{0}^{2\rho(x_t)} \kappa(u)u du}{2(1-\int_{0}^{2\rho(x_t)} \kappa_0(u) du)}\leq\frac{C_1}{\rho(x_0)}\rho(x_t)^2,
\end{eqnarray}
{\small
\begin{eqnarray}\label{eqn:34b}
||z_{t}-x^*||\leq\frac{ \int_{0}^{\rho(x_t)+\rho(y_t)} \kappa(u)u du}{(\rho(x_t)+\rho(y_t))(1-\int_{0}^{2\rho(x_t)} \kappa_0(u) du)}\rho(y_t)\leq\frac{C_1C_2 }{\rho(x_0)\rho(y_0)}\rho(x_t)^3,\nonumber\\
\end{eqnarray}}
{\small
\begin{eqnarray}\label{eqn:34a}
||x_{t+1}-x^*||\leq\frac{ \int_{0}^{\rho(y_t)+\rho(z_t)} \kappa(u)u du}{(\rho(y_t)+\rho(z_t))(1-\int_{0}^{2\rho(y_t)} \kappa_0(u) du)}\rho(z_t)\leq\frac{  C_1^2 C_2 C_3}{\rho(x_0)^2\rho(y_0)\rho(z_0)}\rho(x_t)^5,\nonumber\\
\end{eqnarray}}
where the quantities
{\small
\begin{eqnarray}\label{eqn:36}
&&C_1=\frac{ \int_{0}^{2\rho(x_0)} \kappa(u)u du}{2\rho(x_0)(1-\int_{0}^{2\rho(x_0)} \kappa_0(u) du)},\ C_2=\frac{ \int_{0}^{\rho(x_0)+\rho(y_0)} \kappa(u)u du}{(\rho(x_0)+\rho(y_0))(1-\int_{0}^{2\rho(x_0)} \kappa_0(u) du)},\nonumber\\
&& C_3=\frac{ \int_{0}^{\rho(y_0)+\rho(z_0)} \kappa(u)u du}{(\rho(y_0)+\rho(z_0))(1-\int_{0}^{2\rho(y_0)} \kappa_0(u) du)},
\end{eqnarray}}
are less than $1$. Furthermore,
\begin{eqnarray}\label{eqn:33}
||x_t-x^*||\leq E^{5^t-1}||x_0-x^*||; \ t=1,2,\cdots, E=C_1 C_2\frac{\rho(x_0)^2}{\rho(y_0)\rho(z_0)}.
\end{eqnarray}
\end{thm}
\begin{proof}
By arbitrarily selecting $x_0\in B(x^*, \delta)$, where $\delta$ fulfills the relation $(\ref{eqn:32})$, $C_1$ and $C_2$ are determined by the inequality $(\ref{eqn:36})$ are less than 1. Indeed, because $\kappa$ is monotone, we obtain
\begin{eqnarray}
\left( \frac{1}{f_2^2}\int_{0}^{f_2}-\frac{1}{f_1^2}\int_{0}^{f_1}\right)\kappa(u)udu &=\left(\frac{1}{f_2^2}\int_{f_1}^{f_2}+\left(\frac{1}{f_2^2}-\frac{1}{f_1^2}\right)\int_{0}^{f_1}\right)\kappa(u)udu \nonumber \\
&\geq \kappa(f_1)\left( \frac{1}{f_2^2}\int_{f_1}^{f_2}+\left(\frac{1}{f_2^2}-\frac{1}{f_1^2}\right)\int_{0}^{f_1}\right)udu \nonumber \\
&=\kappa(f_1)\left( \frac{1}{f_2^2}\int_{0}^{f_2}-\frac{1}{f_1^2}\int_{0}^{f_1}\right)udu=0, \nonumber
\end{eqnarray}
for $0<f_1<f_2.$ Thus, $\frac{1}{f^2}\int_{0}^{f}\kappa(u)udu$ is non-decreasing with respect to $f$. Now, we have
\begin{eqnarray}
C_1&=&\frac{ \int_{0}^{2\rho(x_0)} \kappa(u)u du}{2\rho(x_0)^2(1-\int_{0}^{2\rho(x_0)} \kappa_0(u) du)}\rho(x_0)\nonumber\\
&\leq&\frac{ \int_{0}^{2\delta} \kappa(u)u du}{2\delta^2(1-\int_{0}^{2\delta} \kappa_0(u) du)}\rho(x_0) \leq \frac{||x_0-x^*||}{\delta}<1\nonumber,\\
C_2&=&\frac{ \int_{0}^{\rho(x_0)+\rho(y_0)} \kappa(u)u du}{(\rho(x_0)+\rho(y_0))^2(1-\int_{0}^{2\rho(x_0)} \kappa_0(u) du)}(\rho(x_0+\rho(y_0))\nonumber\\
&\leq&\frac{ \int_{0}^{2\delta} \kappa(u)u du}{2\delta^2(1-\int_{0}^{2\delta} \kappa_0(u) du)}(\rho(x_0)+\rho(y_0)) \leq \frac{||x_0-x^*||+||y_0-x^*||}{2\delta}<1,\nonumber\\
C_3&=&\frac{ \int_{0}^{\rho(y_0)+\rho(z_0)} \kappa(u)u du}{(\rho(y_0)+\rho(z_0))^2(1-\int_{0}^{2\rho(y_0)} \kappa_0(u) du)}(\rho(y_0+\rho(z_0))\nonumber\\
&\leq&\frac{ \int_{0}^{2\delta} \kappa(u)u du}{2\delta^2(1-\int_{0}^{2\delta} \kappa_0(u) du)}(\rho(y_0)+\rho(z_0)) \leq \frac{||x_0-x^*||+||y_0-x^*||}{2\delta}<1\nonumber.
\end{eqnarray}
Clearly, if  $x\in B(x^*, \delta)$, then using center Lipschitz condition with the $\kappa$-average and  the relation $(\ref{eqn:32})$, we can write
\begin{eqnarray}\label{eqn:37}
||[G'(x^*)]^{-1}[G'(x)-G'(x^*)]||\leq \int_{0}^{2\rho(x)} \kappa_0(u) du \leq 1.
\end{eqnarray}
Using the Banach Lemma and the following equation,
\begin{eqnarray}
||I-([G'(x^*)]^{-1}G'(x)-I)||^{-1}=||[G'(x)]^{-1}G'(x^*)||, \nonumber
\end{eqnarray}
Using the relation $(\ref{eqn:37})$, we arrive to the following inequality.
\begin{eqnarray}\label{eqn:38}
||[G'(x)]^{-1}G'(x^*)||&\leq& \frac{1}{1-\int_{0}^{2\rho(x)} \kappa_0(u) du}.
\end{eqnarray}
Now, if $x_t\in B(x^*, \delta)$ then we may write from expression $(\ref{eqn:12})$
\begin{eqnarray}\label{eqn:39}
||y_t-x^*||&=&||x_t-x^*-[G'(x_t)]^{-1}G(x_t)|| \nonumber\\
& =&||[G'(x_t)]^{-1}[G'(x_t)(x_t-x^*)-G(x_t)+G(x^*)]||.
\end{eqnarray}
Expanding $G(x_t)$ along $x^*$ from the Taylor series expansion, we may get
{\small
\begin{eqnarray}\label{eqn:310}
G(x^*)-G(x_t)+t'(x_t)(x_t-x^*)=G'(x^*)\int_{0}^{1}[G'(x^*)]^{-1}[G'(x_t)-G'(x^\tau)]d\tau(x_t-x^*). \nonumber\\
\end{eqnarray}}
Additionally, from the expression $(\ref{eqn:23})$ and combining the equations $(\ref{eqn:39})$ and  $(\ref{eqn:310})$, we reach to
{\small
\begin{eqnarray}\label{eqn:311}
||y_t-x^*||&\leq&||[G'(x_t)]^{-1}G'(x^*)||.||\int_{0}^{1}[G'(x^*)]^{-1}[G'(x_t)-G'(x^\tau)]d\tau||.||(x_t-x^*)||\nonumber\\
&\leq& \frac{1}{1-\int_{0}^{2\rho(x_t)} \kappa_0(u) du}\int_{0}^{1}\int_{2\tau\rho(x_n)}^{2\rho(x_n)}\kappa(u)du\rho(x_t)d\tau.
\end{eqnarray}}
As a result of Lemma $(\ref{lm:21})$ and  the above expression, The first inequality of expression  $(\ref{eqn:34})$ can be obtained. By similar analogy and using the second sub-step of the scheme  $(\ref{eqn:12})$, we can write
{\small
\begin{eqnarray}\label{eqn:312a}
||z_{t}-x^*||&\le&||[G'(x_t)]^{-1}G'(x^*)||.||\int_{0}^{1}[G'(x^*)]^{-1}[G'(x_t)-G'(y^\tau)]d\tau||.||(y_t-x^*)||\nonumber\\
&\le&\frac{1}{1-\int_{0}^{2\rho(x_t)} \kappa_0(u) du}\int_{0}^{1}\int_{\tau(\rho(x_t)+\rho(y_t))}^{\rho(x_t)+\rho(y_t)}\kappa(u)du\rho(y_t)d\tau.
\end{eqnarray}}
Using Lemma $(\ref{lm:21})$  and above expression, we can get the first inequality of expression $(\ref{eqn:34b})$.
Simultaneously, re-writing the last sub-step of the scheme  $(\ref{eqn:12})$, we achieve
{\small
\begin{eqnarray}\label{eqn:312}
||x_{t+1}-x^*||&\le&||[G'(y_t)]^{-1}G'(x^*)||.||\int_{0}^{1}[t'(x^*)]^{-1}[G'(y_t)-G'(z^\tau)]d\tau||.||(z_t-x^*)||\nonumber\\
&\le&\frac{1}{1-\int_{0}^{2\rho(y_t)} \kappa_0(u) du}\int_{0}^{1}\int_{\tau(\rho(y_t)+\rho(z_t))}^{\rho(y_t)+\rho(z_t)}\kappa(u)du\rho(z_n)d\tau.
\end{eqnarray}}
Using Lemma $(\ref{lm:21})$  and above expression, we can get the first inequality of expression $(\ref{eqn:34a})$.
Furthermore, $\rho(x_t)$ and $\rho(y_t)$ are decreasing monotonically, therefore for all $t=0,1,...$, we have
\begin{eqnarray}
||y_t-x^*||&\leq& \frac{ \int_{0}^{2\rho(x_t)} \kappa(u)u du}{2(1-\int_{0}^{2\rho(x_t)} \kappa_0(u) du)}\nonumber\\
&\leq&\frac{ \int_{0}^{2\rho(x_0)} \kappa(u)u du}{2\rho(x_0)^2(1-\int_{0}^{2\rho(x_t)} \kappa_0(u) du)}2\rho(x_t)^2\leq \frac{C_1}{\rho(x_0)}\rho(x_t)^2.\nonumber
\end{eqnarray}
Also, by manipulating first inequality of expression $(\ref{eqn:34b})$, we have
{\small
\begin{eqnarray}\label{eqn:313}
||z_{t}-x^*||&\leq& \frac{ \int_{0}^{\rho(x_t)+\rho(y_t)} \kappa(u)u du}{(\rho(x_t)+\rho(y_t))^2(1-\int_{0}^{2\rho(x_t)} \kappa_0(u) du)}\rho(y_t).[\rho(x_t)+\rho(y_t)]\nonumber\\
&\leq&\frac{C_2}{\rho(x_0)+\rho(y_0)}[\rho(x_t)+\rho(y_t)] \rho(y_t) \\
& \leq & \frac{ C_1 C_2}{\rho(x_0)\rho(y_0)}\rho(x_t)^3. \nonumber
\end{eqnarray}}
Lastly, simplifying the first result in the inequality of expression $(\ref{eqn:34a})$, we have
{\small
\begin{eqnarray*}\label{eqn:315}
||x_{t+1}-x^*||&\leq& \frac{ \int_{0}^{\rho(y_t)+\rho(z_t)} \kappa(u)u du}{(\rho(y_t)+\rho(z_t))^2(1-\int_{0}^{2\rho(y_t)} \kappa_0(u) du)}\rho(z_t).[\rho(y_t)+\rho(z_t)]\nonumber\\
&\leq&\frac{C_3}{\rho(y_0)+\rho(z_0)}[\rho(y_t)\rho(z_t)+\rho(z_t)^2] \leq\frac{ C_1^2 C_2 C_3 }{\rho(x_0^2)\rho(y_0)\rho(z_0)}\rho(x_t)^5. \nonumber\\
\end{eqnarray*}}
Thus, we have derived all the expressions of inequalities  $(\ref{eqn:34})$, $(\ref{eqn:34b})$ and $(\ref{eqn:34a})$.
Now,  we use mathematical induction to prove the inequality $(\ref{eqn:33})$.
For $t=0$, the inequality $(\ref{eqn:34a})$ gives
\begin{eqnarray*}\label{eqn:313a}
||x_{1}-x^*||\leq \frac{ C_1^2 C_2 C_3 }{\rho(x_0)^2\rho(y_0)\rho(z_0)}\rho(x_0)^5.
\end{eqnarray*}
Multiplying the numerator and denominator by $\rho(y_0)^3\rho(z_0)^3$ and using the inequalities $(\ref{eqn:34b})$ and $(\ref{eqn:313})$  for $t=0$ and then doing some mathematical manipulations, the above inequality may be re-written as
\begin{eqnarray*}\label{eqn:313b}
||x_{1}-x^*||
&\leq&\left[C_1^2 C_2\frac{\rho(x_0)^2}{\rho(y_0)\rho(z_0)}\right]^{4} \rho(x_0) \nonumber\\
&\leq&E^{(5-1)} \rho(x_0). 
\end{eqnarray*}
Thus the expression $(\ref{eqn:33})$ is true for  $t=1$. Now, suppose the inequality $(\ref{eqn:33})$ is true for some integer $t > 1$.
%
%
Again multiplying the numerator and denominator by $\rho(y_0)^3\rho(z_0)^3$ and using the inequalities $(\ref{eqn:33})$ for $t=t$,  $(\ref{eqn:34b})$ and $(\ref{eqn:313})$ for $t=0$,  the above inequality preserves the form
\begin{eqnarray*}
||x_{t+1}-x^*||
&\leq&\left[C_1^2 C_2\frac{\rho(x_0)^2}{\rho(y_0)\rho(z_0)}\right]^{5^{t+1}-1} \rho(x_0) \nonumber\\
&\leq&E^{(5^{t+1}-1)} \rho(x_0) .
\end{eqnarray*}
Thus we reach at completion of the proof.
\end{proof}

\section{\bf The uniqueness ball for the optimal solution}
We will show the uniqueness theorem for the Newton-Traub approach $(\ref{eqn:12})$, using the center Lipschitz condition.
\begin{thm}\label{th:41}
Suppose that $G(x^*)=0$, $G$ has a continuous derivative in  $B(x^*, \delta)$, $[G'(x^*)]^{-1}$ exists and $[G'(x^*)]^{-1}G'$ satisfies  $(\ref{eqn:24})$. Let $\delta$ satisfy the relation
\begin{eqnarray}\label{eqn:42}
\frac{ \int_{0}^{2\delta} \kappa_0(u)(2\delta-u) du}{2\delta}\leq 1.
\end{eqnarray}
Then, the equation  $G(x)=0$ has a unique solution $x^*$ in $B(x^*, \delta)$.
\end{thm}
\begin{proof}
By arbitrarily choosing $y^*\in B(x^*, r)$, $y^*\neq x^*$ and evaluating the iteration, we get
\begin{eqnarray}\label{eqn:43}
||y^*-x^*||&=&||y^*-x^*-[G'(x^*)]^{-1}G(y^*)||. \nonumber\\
&=&||[G'(x^*)]^{-1}[G'(x^*)(y^*-x^*)-G(y^*)+G(x^*)]||\nonumber\\
\end{eqnarray}
Expanding $G(y^*)$  along $\ x^*$ from the Taylor's expansion, we have
\begin{eqnarray}\label{eqn:44}
G(x^*)-G(y^*)+G'(x^*)(y^*-x^*)=\int_{0}^{1}[G'(x^*)]^{-1}[G'(y^{*\tau})-G'(x^*)]d\tau(y^*-x^*),\nonumber\\
\end{eqnarray}
Using the expression $(\ref{eqn:24})$ and combining the inequalities $(\ref{eqn:43})$ and $(\ref{eqn:44})$, we may write
{\small
\begin{eqnarray}\label{eqn:45}
||y^*-x^*||&\leq&||[G'(x^*)]^{-1}G'(x^*)||.||\int_{0}^{1}[G'(x^*)]^{-1}[G'(y^{*\tau})-G'(x^*)]d\tau||.||(y^*-x^*)||\nonumber\\
&\leq& \int_{0}^{1}\int_{0}^{2\tau\rho(y^*)}L_0(u)du\rho(y^*)d\tau.
\end{eqnarray}}
As a result of Lemma  $(\ref{lm:21})$ and  expression $(\ref{eqn:45})$, we obtain 
\begin{eqnarray}\label{eqn:46}
||y^*-x^*||&\leq&\frac{1}{2\rho(y^*)}\int_{0}^{2\rho(y^*)}\kappa_0(u)[2\rho(y^*)-u]du(y^*-x^*)\nonumber\\
&\leq& \frac{ \int_{0}^{2\delta} \kappa_0(u)(2\delta-u) du}{2\delta}\rho(y^*)\leq||y^*-x^*||.
\end{eqnarray}
But this contradicts our hypotheses. Thus, we see that $y^*=x^*$. This concludes the theorem's proof.
\end{proof}
 

Specifically, assuming that $\kappa$ and $\kappa_0$ are constants, we have the Corollaries $(\ref{cr:51})$ and $(\ref{cr:52})$  derived from Theorems $(\ref{th:31})$ and $(\ref{th:41})$, respectively.
								
\begin{Corollary}\label{cr:51}
Suppose that $x^*$ satisfies $G(x^*)=0$, $G$ has a continuous derivative in  $B(x^*, \delta)$, $[G'(x^*)]^{-1}$ exists and $[G'(x^*)]^{-1}G'$ satisfies $(\ref{eqn:21})$ and $(\ref{eqn:22})$. Let $\delta$ satisfy the relation
\begin{eqnarray}\label{eqn:48}
\delta&=&\frac{1}{2\kappa_0+\kappa}.
\end{eqnarray}
Then, the three-step Newton-Traub method $ (\ref{eqn:12})$ is convergent for all  \ $x_0\in B(x^*, \delta)$ and 
\begin{eqnarray*}\label{eqn:410}
||y_t-x^*||\leq\frac{C_1}{\rho(x_0)}\rho(x_t)^2,
\end{eqnarray*}
\begin{eqnarray*}\label{eqn:411a}
||z_{t}-x^*||\leq\frac{C_1 C_2 }{\rho(x_0)\rho(y_0)}\rho(x_t)^3,
\end{eqnarray*}
\begin{eqnarray*}\label{eqn:411}
||x_{t+1}-x^*||\leq\frac{  C_1^2 C_2 C_3}{\rho(x_0)^2\rho(y_0)\rho(z_0)}\rho(x_t)^5,
\end{eqnarray*}
where the quantities
\begin{eqnarray}\label{eqn:412}
C_1=\frac{\kappa\rho(x_0)}{1-2\kappa_0\rho(x_0)}, C_2=\frac{\kappa(\rho(x_0)+\rho(y_0))}{2(1-2\kappa_0\rho(x_0))},
C_3=\frac{\kappa(\rho(y_0)+\rho(z_0))}{2(1-2\kappa_0\rho(y_0))}, \nonumber\\
\end{eqnarray}
are less than $1$. 
\end{Corollary}
\begin{Corollary}\label{cr:52}
Suppose that $x^*$ satisfies $G(x^*)=0$, $G$ has a continuous derivative in  $B(x^*, \delta)$, $[G'(x^*)]^{-1}$ exists and $[G'(x^*)]^{-1}G'$ satisfies the assumption $(\ref{eqn:22})$. Let $\delta$ fulfill the condition
\begin{eqnarray}\label{eqn:414}
\delta&=&\frac{1}{\kappa_0}.
\end{eqnarray}
Then, the equation  $G(x)=0$ has a unique solution $x^*$ in  $B(x^*, \delta)$. Moreover, the ball radius $\delta$ depends only on $\kappa_0$.
\end{Corollary}       
Following that, we will apply our fundamental theorems to some specific function $\kappa$ and find the following corollaries:
\begin{Corollary}\label{cr:56}
Suppose that $x^*$ satisfies $G(x^*)=0$, $G$ has a continuous derivative in  $B(x^*, \delta)$, $[G'(x^*)]^{-1}$ exists and $[G'(x^*)]^{-1}G'$ satisfies $(\ref{eqn:23})$, $(\ref{eqn:24})$ where given fixed positive constants $\gamma$, $\kappa>0$ and $\kappa_0>0$ with $\kappa(u)=\gamma +\kappa u$ and $\kappa_0(u)=\gamma +\kappa_0u$ i.e.:
\begin{eqnarray}\label{eqn:561}
||[G'(x^*)]^{-1}(G'(x)-G'(y^\tau))|| &\leq & \gamma(1-\tau)(||x-x^*||+||y-x^*||)\nonumber\\
&&+\frac{\kappa}{2}(1-\tau^2)(||x-x^*||+||y-x^*||)^2, \nonumber\\
\end{eqnarray}
and 
\begin{eqnarray}\label{eqn:561a}
||[G'(x^*)]^{-1}(G'(x)-G'(x^*))||\leq 2||x-x^*||(\gamma+\kappa_0||x-x^*||), 
\end{eqnarray}
$\forall \ x,\ y \in G(x^*, \delta), 0\leq\tau\leq 1$, where $y^\tau=x^*+\tau(y-x^*)$, $\rho(x)=||x-x^*||$. Let $\delta$ satisfy the relation
\begin{eqnarray}\label{eqn:562}
\delta&=&\frac{-3\gamma+\sqrt{9\gamma^2+(16/3)\kappa+8\kappa_0}}{8/3\kappa+4\kappa_0}\ and\ 9\gamma^2+(16/3)\kappa+8\kappa_0\geq0.\nonumber\\ 
\end{eqnarray}
Then, the three-step Newton-Traub method $ (\ref{eqn:12})$ is convergent for all  \ $x_0\in B(x^*, \delta)$ and 
\begin{eqnarray*}\label{eqn:563}
||y_t-x^*||\leq\frac{C_1}{\rho(x_0)}\rho(x_t)^2,
\end{eqnarray*}
\begin{eqnarray*}\label{eqn:564}
||z_{t}-x^*||\leq\frac{ C_1 C_2}{\rho(x_0)\rho(y_0)}\rho(x_t)^3,
\end{eqnarray*}
\begin{eqnarray*}\label{eqn:564a}
||x_{t+1}-x^*||\leq\frac{  C_1^2 C_2 C_3}{\rho(x_0)^2\rho(y_0)\rho(z_0)}\rho(x_t)^5,
\end{eqnarray*}
where the quantities
\begin{eqnarray}\label{eqn:565}
C_1&=&\frac{\rho(x_0)[\gamma+4/3\kappa\rho(x_0)]}{[1-2\gamma \rho(x_0)-2\kappa_0\rho(x_0)^2]},\nonumber\\
C_2&=&\frac{\rho(x_0)+\rho(y_0)[\gamma/2+\kappa/3(\rho(x_0)+\rho(y_0)]}{[1-2\gamma \rho(x_0)-2\kappa_0\rho(x_0)^2]},\nonumber\\
C_3&=&\frac{\rho(y_0)+\rho(z_0)[\gamma/2+\kappa/3(\rho(y_0)+\rho(z_0)]}{[1-2\gamma \rho(y_0)-2\kappa_0\rho(y_0)^2]},
\end{eqnarray}
are less than $1$. 
\end{Corollary}
\begin{Corollary}\label{cr:57}
Suppose that $x^*$ satisfies $G(x^*)=0$, $G$ has a continuous derivative in  $B(x^*, \delta)$, $[G'(x^*)]^{-1}$ exists and $[G'(x^*)]^{-1}G'$ satisfies  $(\ref{eqn:24})$ where given fixed positive constants $\gamma$ and $\kappa_0>0$ with $\kappa_0(u)=\gamma +\kappa_0u$ i.e.:
\begin{eqnarray}\label{eqn:571}
||[G'(x^*)]^{-1}(G'(x)-G'(x^*))||\leq 2||x-x^*||(\gamma+\kappa_0||x-x^*||); \forall \ x \in B(x^*, \delta), \nonumber\\
\end{eqnarray}
where $\rho(x)=||x-x^*||$. Let $\delta$ satisfy the relation
\begin{eqnarray}\label{eqn:572}
\delta&=&\frac{2\gamma-\sqrt{4\gamma^2-(16/3)\kappa_0}}{(8/3)\kappa_0}\ and\ 4\gamma^2-(16/3)\kappa_0\geq0.
\end{eqnarray}
Then, the equation  $G(x)=0$ has a unique solution $x^*$ in  $B(x^*, \delta)$. Moreover, the ball radius $\delta$ depends only on $\kappa_0$ and $\gamma$.
\end{Corollary}

\section{\bf Local Convergence under weak $\kappa$-average}
This section presents the findings of a re-examination of the requirements and radius of convergence of the considered scheme, which were previously stated in the first theorem, although $\kappa$ is not regarded a non-decreasing function. The convergence order has been observed to be decreasing. The second theorem in this section yields a result identical to theorem $(\ref{th:51})$, but with the center Lipschitz condition assumed.
\begin{thm}\label{th:51}
Suppose that $G(x^*)=0$, $G$ has a continuous derivative in  $B(x^*, \delta)$, $[G'(x^*)]^{-1}$ exists and $[G'(x^*)]^{-1}G'$ satisfies the assumptions $(\ref{eqn:23})$, $(\ref{eqn:24})$, $\kappa_0$ and $\kappa$ are positive integrable. Let $\delta$ satisfy
\begin{eqnarray}\label{eqn:52}
\int_{0}^{2\delta}(\kappa(u)+\kappa_0(u))du\leq 1.
\end{eqnarray}
Then, the three-step Newton-Traub method $(\ref{eqn:12})$ is convergent for all $x_0\in B(x^*, \delta)$ and 
\begin{eqnarray}\label{eqn:54}
||y_t-x^*||&\leq&\frac{ \int_{0}^{2\rho(x_t)} \kappa(u)u du}{2(1-\int_{0}^{2\rho(x_t)} \kappa_0(u) du)}\leq q_1\rho(x_t),
\end{eqnarray}
\begin{eqnarray}\label{eqn:54b}
||z_{t}-x^*||&\leq&\frac{ \int_{0}^{\rho(x_t)+\rho(y_t)} \kappa(u)u du}{(\rho(x_t)+\rho(y_t))(1-\int_{0}^{2\rho(x_t)} \kappa_0(u) du)}\rho(y_t)\leq q_2 q_1\rho(x_t),\nonumber\\
\end{eqnarray}
\begin{eqnarray}\label{eqn:54a}
||x_{t+1}-x^*||&\leq&\frac{ \int_{0}^{\rho(y_t)+\rho(z_t)} \kappa(u)u du}{(\rho(y_t)+\rho(z_t))(1-\int_{0}^{2\rho(y_t)} \kappa_0(u) du)}\rho(z_t)\leq q_3 q_2 q_1\rho(x_t),\nonumber\\
\end{eqnarray}
where the quantities
\begin{eqnarray}\label{eqn:55}
&& q_1=\frac{ \int_{0}^{2\rho(x_0)} \kappa(u)du}{1-\int_{0}^{2\rho(x_0)} \kappa_0(u) du},\ q_2=\frac{ \int_{0}^{\rho(x_0)+\rho(y_0)} \kappa(u)du}{1-\int_{0}^{2\rho(x_0)} \kappa_0(u) du},\nonumber\\
&& q_3=\frac{ \int_{0}^{\rho(y_0)+\rho(z_0)} \kappa(u)du}{1-\int_{0}^{2\rho(y_0)} \kappa_0(u) du},
\end{eqnarray}
are less than $1$. Moreover,
\begin{eqnarray}\label{eqn:53}
||x_t-x^*||\leq (q_1q_2 q_3)^{t}||x_0-x^*||,\ t=1,2,....
\end{eqnarray}
Furthermore, suppose that the function $\kappa_{a}$ is defined by 
\begin{eqnarray}\label{eqn:56}
\kappa_a(P)=P^{1-a}\kappa(P),
\end{eqnarray}
is non-decreasing for some $a$ with $0\le a\le1$ and $\delta$ satisfies
\begin{equation}\label{eqn:57}
\frac{1}{2\delta}\int_{0}^{2\delta}(2\delta\kappa_0(u)+u\kappa(u))du\le1.
\end{equation}
Then, the three-step Newton-Traub method $(\ref{eqn:12})$ is convergent for all \ $x_0\in B(x^*, \delta)$ and 
\begin{eqnarray}\label{eqn:58}
||x_t-x^*||\leq E^{(1+3a+a^2)^t-1}||x_0-x^*||,\ t=1,2,\cdots, E=C_1^2 C_2\frac{\rho(x_0)^2}{\rho(y_0)\rho(z_0)},\nonumber\\
\end{eqnarray}
where the quantities $C_1$ and $C_2$ are defined in the equation $(\ref{eqn:36})$ and are less than $1$.
%
%
\end{thm}
\begin{proof}
We can show that the quantities $q_1$, $q_2$ and $q_3$ described by the equation  $(\ref{eqn:55})$ as less than 1 by arbitrarily picking $x_0\in B(x^*, \delta)$, where $\delta$ fulfills the relation $(\ref{eqn:52})$. Indeed, because $\kappa$ is positive integrable function, we may obtain
\begin{eqnarray}
q_1&=&\frac{ \int_{0}^{2\rho(x_0)} \kappa(u) du}{1-\int_{0}^{2\rho(x_0)} \kappa_0(u) du}
\leq\frac{ \int_{0}^{2\delta} \kappa(u) du}{1-\int_{0}^{2\delta} \kappa_0(u) du}
<1\nonumber,\\
q_2&=&\frac{ \int_{0}^{\rho(x_0)+\rho(y_0)} \kappa(u) du}{1-\int_{0}^{2\rho(x_0)} \kappa_0(u) du}
\leq\frac{ \int_{0}^{2\delta} \kappa(u) du}{1-\int_{0}^{2\delta} \kappa_0(u) du}
<1,\nonumber\\
q_3&=&\frac{ \int_{0}^{\rho(y_0)+\rho(z_0)} \kappa(u) du}{1-\int_{0}^{2\rho(y_0)} \kappa_0(u) du}
\leq\frac{ \int_{0}^{2\delta} \kappa(u) du}{1-\int_{0}^{2\delta} \kappa_0(u) du}
<1\nonumber.
\end{eqnarray}
Obviously, if  $x \in B(x^*, \delta)$, then using center Lipschitz condition with the $\kappa$-average and the relation $(\ref{eqn:52})$, we have
\begin{eqnarray}\label{eqn:510}
||[G'(x^*)]^{-1}[G'(x)-G'(x^*)]||\leq \int_{0}^{2\rho(x)} \kappa_0(u) du; \forall \ x \in B(x^*, \delta)
\leq1. \nonumber\\
\end{eqnarray}
Using the Banach Lemma and the following equation
\begin{eqnarray}
||I-([G'(x^*)]^{-1}G'(x)-I)||^{-1}=||[G'(x)]^{-1}G'(x^*)||, \nonumber
\end{eqnarray}
we come to following inequality using the relation $(\ref{eqn:510})$
\begin{eqnarray}\label{eqn:511}
||[G'(x)]^{-1}G'(x^*)||&\leq& \frac{1}{1-\int_{0}^{2\rho(x)} \kappa_0(u) du}.
\end{eqnarray}
Hence, if $x_t\in B(x^*, \delta)$, then we may write from first sub-step of scheme $(\ref{eqn:12})$
\begin{eqnarray}\label{eqn:512}
||y_t-x^*||&=&||x_t-x^*-[G'(x_t)]^{-1}G(x_t)||\nonumber\\
&=&||[G'(x_t)]^{-1}[G'(x_t)(x_t-x^*)-G(x_t)+G(x^*)]||. \nonumber\\
\end{eqnarray}
Expanding $G(x_t)$  around  $x^*$ from the Taylor's expansion, it can written as
\begin{eqnarray}\label{eqn:513}
G(x^*)-G(x_t)+G'(x_t)(x_t-x^*)=G'(x^*)\int_{0}^{1}[G'(x^*)]^{-1}[G'(x_t)-G'(x_t^\tau)]d\tau(x_t-x^*).\nonumber\\
\end{eqnarray}
Following the assumptions  $(\ref{eqn:23})$ and combining the inequalities $(\ref{eqn:512})$ and  $(\ref{eqn:513})$, we may write
\begin{eqnarray}\label{eqn:514}
||y_t-x^*||&\leq&||[G'(x_t)]^{-1}G'(x^*)||.||\int_{0}^{1}[G'(x^*)]^{-1}[G'(x_t)-G'(x_t^\tau)]d\tau||.||(x_t-x^*)||\nonumber\\
&\leq& \frac{1}{\int_{0}^{2\rho(x_t)} \kappa_0(u) du}\int_{0}^{1}\int_{2\tau\rho(x_t)}^{2\rho(x_t)}\kappa(u)du\rho(x_t)d\tau.
\end{eqnarray}
Using the results of  Lemma $(\ref{lm:21})$ and the inequality $(\ref{eqn:511})$ in the above expression we can obtain the first inequality of $(\ref{eqn:54})$. By similar analogy for the second sub-step of the scheme $(\ref{eqn:12})$, we can write
{\small
\begin{eqnarray}\label{eqn:515a}
||z_{t}-x^*||&\le&||[G'(x_t)]^{-1}G'(x^*)||.||\int_{0}^{1}[G'(x^*)]^{-1}[G'(x_t)-G'(y_t^\tau)]d\tau||.||(y_t-x^*)||\nonumber\\
&\le&\frac{1}{\int_{0}^{2\rho(x_t)} \kappa_0(u) du}\int_{0}^{1}\int_{\tau(\rho(x_t)+\rho(y_t))}^{\rho(x_t)+\rho(y_t)}\kappa(u)du\rho(y_t)d\tau. 
\end{eqnarray}}
This way, we get the first inequality of $ (\ref{eqn:54b})$ using  the Lemma $(\ref{lm:21})$ in the above expression. Similarly, rewriting the last sub-step of the scheme  $(\ref{eqn:12})$, we achieve
{\small
\begin{eqnarray}\label{eqn:515}
||x_{t+1}-x^*||&\le&||[G'(y_t)]^{-1}G'(x^*)||.||\int_{0}^{1}[G'(x^*)]^{-1}[G'(y_t)-G'(z_t^\tau)]d\tau||.||(z_t-x^*)||\nonumber\\
&\le&\frac{1}{\int_{0}^{2\rho(y_t)} \kappa_0(u) du}\int_{0}^{1}\int_{\tau(\rho(y_t)+\rho(z_t))}^{\rho(y_t)+\rho(z_t)}\kappa(u)du\rho(z_t)d\tau. 
\end{eqnarray}}
Using Lemma $(\ref{lm:21})$ in the above expression, we can get the first inequality of $ (\ref{eqn:54a})$. Moreover, $\rho(x_t)$, $\rho(y_t)$ and $\rho(z_t)$ are decreasing monotonically, therefore for all $t=0,1,...,$ we have
\begin{eqnarray}
||y_t-x^*||\leq \frac{ \int_{0}^{2\rho(x_t)} \kappa(u)u du}{2(1-\int_{0}^{2\rho(x_t)} \kappa_0(u) du)}
\leq \frac{ \int_{0}^{2\rho(x_0)} \kappa(u) du}{1-\int_{0}^{2\rho(x_t)} \kappa_0(u) du}\rho(x_t)
\leq q_1\rho(x_t).\nonumber
\end{eqnarray}
Using the second inequality of expression $(\ref{eqn:54})$,
 we arrive at
\begin{eqnarray}\label{eqn:517}
||z_{t}-x^*||&\leq& \frac{ \int_{0}^{\rho(x_t)+\rho(y_t)} \kappa(u)u du}{(\rho(x_t)+\rho(y_t))(1-\int_{0}^{2\rho(x_t)} \kappa_0(u) du)}\rho(y_t)\nonumber\\
&\leq&\frac{ \int_{0}^{\rho(x_0)+\rho(y_0)} \kappa(u) du}{(1-\int_{0}^{2\rho(x_0)} \kappa_0(u) du)}\rho(y_t) \leq q_2q_1\rho(x_t).
\end{eqnarray}
Next, with the help of inequality of expression $(\ref{eqn:54b})$,
we are able to reach at
\begin{eqnarray}\label{eqn:517a}
||x_{t+1}-x^*||&\leq& \frac{ \int_{0}^{\rho(y_t)+\rho(z_t)} \kappa(u)u du}{(\rho(y_t)+\rho(z_t))(1-\int_{0}^{2\rho(y_t)} \kappa_0(u) du)}\rho(z_t)\nonumber\\
&\leq&\frac{ \int_{0}^{\rho(y_0)+\rho(z_0)} \kappa(u) du}{1-\int_{0}^{2\rho(y_0)} \kappa_0(u) du}\rho(z_t) \leq q_3 q_2 q_1\rho(x_t).
\end{eqnarray}
Therefore, the inequality $(\ref{eqn:53})$ may be easily deduced from the inequality $(\ref{eqn:517a})$. Furthermore, if the function $\kappa_a$ defined by the relation $(\ref{eqn:56})$ is non-decreasing for some $a$ with $0\le a\le1$ and $\delta$ is determined by the expression $(\ref{eqn:57})$, the first inequality of expression $(\ref{eqn:54})$ and Lemma $(\ref{lm:22})$ imply that
\begin{eqnarray}\label{eqn:4444c}
||y_t-x^*||&\leq&\frac{ \varphi_{1,a}(2\rho(x_t))2^a}{1-\int_{0}^{2\rho(x_t)} \kappa_0(u) du}\rho(x_t)^{a+1}\nonumber\\
&\le&\frac{ \varphi_{1,a}(2\rho(x_0))2^a}{1-\int_{0}^{2\rho(x_t)} \kappa_0(u) du}\rho(x_t)^{a+1}=\frac{C_1}{\rho(x_0)^a}\rho(x_t)^{a+1}\
\end{eqnarray}
Moreover, from the first part of inequality $(\ref{eqn:54b})$ and Lemma $(\ref{lm:22})$, we may write
 \begin{eqnarray} \label{eqn:4444b}
||z_{t}-x^*||&\leq&\frac{ \varphi_{1,a}(\rho(x_t)+\rho(y_t))(\rho(x_t)+\rho(y_t))^a}{1-\int_{0}^{2\rho(x_t)} \kappa_0(u) du}\rho(y_t),\nonumber\\
&\leq&\frac{ \varphi_{1,a}(\rho(x_0)+\rho(y_0))(\rho(x_t)+\rho(y_t))^a}{1-\int_{0}^{2\rho(x_t)} \kappa_0(u) du}\rho(y_t),\nonumber\\
&\leq&\frac{C_2}{(\rho(x_0)+\rho(y_0))^a}(\rho(x_t)+\rho(y_t))^a\rho(y_t),\\
&\leq&\frac{C_1 C_2}{\rho(x_0)^a\rho(y_0)^a}\rho(x_t)^{2a+1}. \nonumber
\end{eqnarray}
Similarly, from the first inequality of $(\ref{eqn:54a})$ and Lemma $(\ref{lm:22})$, we arrive at
 \begin{eqnarray}\label{eqn:4444a}
||x_{t+1}-x^*||&\leq&\frac{ \varphi_{1,a}(\rho(y_t)+\rho(z_t))(\rho(y_t)+\rho(z_t))^a}{1-\int_{0}^{2\rho(y_t)} \kappa_0(u) du}\rho(z_t)\nonumber\\
&\leq&\frac{ \varphi_{1,a}(\rho(y_0)+\rho(z_0))(\rho(y_t)+\rho(z_t))^a}{1-\int_{0}^{2\rho(y_t)} \kappa_0(u) du}\rho(z_t)\nonumber\\
&=&\frac{ C_1^{a+1} C_2 C_3 }{\rho(x_0)^{a+a^2}\rho(y_0)^a\rho(z_0)^a}\rho(x_t)^{(a+1)^2+a},
\end{eqnarray}
where $C_1<1$, $C_2<1$ and $C_3<1$ are determined by the expression $(\ref{eqn:36})$. 
Now,  we use mathematical induction to prove the inequality $(\ref{eqn:58})$.
For $t=0$, the above inequality becomes
\begin{eqnarray*}\label{eqn:313a}
||x_{1}-x^*||\leq \frac{ C_1^{1+a} C_2 C_3 }{\rho(x_0)^{a^2+a}\rho(y_0)^a\rho(z_0)^a}\rho(x_0)^{(a+1)^2+a}.
\end{eqnarray*}
Multiplying the numerator and denominator by $\rho(y_0)^{a^2+2a}\rho(z_0)^{a^2+2a}$ and using the inequality $(\ref{eqn:4444b})$ for $t=0$ and then doing the calculations, we can get
\begin{eqnarray}\label{eqn:313b}
||x_{1}-x^*||
&\leq&\left[C_1^2 C_2\frac{\rho(x_0)^2}{\rho(y_0)\rho(z_0)}\right]^{a^2+3a} \rho(x_0) \nonumber\\
&\leq&E^{((1+3a+a^2)-1)} \rho(x_0). 
\end{eqnarray}
Thus the expression $(\ref{eqn:58})$ is true for  $t=1$. Now, suppose the inequality $(\ref{eqn:58})$ is true for some integer $t > 1$.
%
%
Again multiplying in the numerator and denominator of the inequality $(\ref{eqn:4444a})$ by $\rho(y_0)^{a^2+2a}\rho(z_0)^{a^2+2a}$ and then using  the inequalities $(\ref{eqn:58})$ for $t=t$,  $(\ref{eqn:4444b})$, $(\ref{eqn:4444c})$ after re-arranging the terms, the above inequality preserves the form
\begin{eqnarray*}
||x_{t+1}-x^*||
&\leq&\left[C_1^2 C_2\frac{\rho(x_0)^2}{\rho(y_0)\rho(z_0)}\right]^{(1+3a+a^2)^{t+1}-1} \rho(x_0) \nonumber\\
&\leq&E^{(1+3a+a^2)^{t+1}-1)} \rho(x_0).
\end{eqnarray*}
which shows that the result is true for $t=t+1$, hence the inequality $(\ref{eqn:58})$ is true for all natural number 
and consequently $x_t$ converges to $x^*$. Hence, the proof is finished.
\end{proof}
\begin{thm}\label{th:52}
Suppose that $G(x^*)=0$, $G$ has a continuous derivative in  $B(x^*, \delta)$, $[G'(x^*)]^{-1}$ exists and $[G'(x^*)]^{-1}G'$ satisfies the assumption $(\ref{eqn:24})$ and $\kappa_0$ is positive integrable function. Let $\delta$ satisfy
\begin{eqnarray}\label{eqn:522}
\int_{0}^{2\delta}\kappa_0(u)du\leq \frac{1}{3}.
\end{eqnarray}
Then, the three-step Newton-Traub method $(\ref{eqn:12})$ is convergent for all  \\ \ $x_0\in B(x^*, \delta)$ and 
\begin{eqnarray}\label{eqn:520a}
||y_t-x^*||&\leq&\frac{2 \int_{0}^{2\rho(x_t)} \kappa_0(u) du}{1-\int_{0}^{2\rho(x_t)} \kappa_0(u) du}\rho(x_t)\leq q_1\rho(x_t), \nonumber\\
||z_{t}-x^*||&\leq&\frac{ \int_{0}^{2\rho(x_t)} \kappa_0(u) du+ \int_{0}^{2\rho(y_t)} \kappa_0(u) du}{1-\int_{0}^{2\rho(x_t)} \kappa_0(u) du}\rho(y_t)\leq q_2 q_1\rho(x_t),\nonumber\\
||x_{t+1}-x^*||&\leq&\frac{ \int_{0}^{2\rho(y_t)} \kappa_0(u) du+ \int_{0}^{2\rho(z_t)} \kappa_0(u) du}{1-\int_{0}^{2\rho(y_t)} \kappa_0(u) du}\rho(z_t)\leq q_3 q_2 q_1\rho(x_t),\nonumber\\
\end{eqnarray}
where the quantities
\begin{eqnarray}\label{eqn:524}
&& q_1=\frac{2 \int_{0}^{2\rho(x_0)} \kappa_0(u)du}{1-\int_{0}^{2\rho(x_0)} \kappa_0(u) du},\ q_2=\frac{ \int_{0}^{2\rho(x_0)} \kappa_0(u)du+\int_{0}^{2\rho(y_0)} \kappa_0(u)du}{1-\int_{0}^{2\rho(x_0)} \kappa_0(u) du}, \nonumber\\
&& q_3=\frac{ \int_{0}^{2\rho(y_0)} \kappa_0(u)du+\int_{0}^{2\rho(z_0)} \kappa_0(u)du}{1-\int_{0}^{2\rho(y_0)} \kappa_0(u) du},
\end{eqnarray}
are less than $1$. Moreover,
\begin{eqnarray}\label{eqn:523}
||x_t-x^*||\leq (q_1q_2q_3)^{t}||x_0-x^*||,\ t=1,2,... .
\end{eqnarray}
Furthermore, suppose that the function $\kappa_{a}$ defined by the relation $(\ref{eqn:56})$ is non-decreasing for some $a$ with $0\le a\le1$, then
\begin{eqnarray}\label{eqn:525}
||x_t-x^*||\leq F^{(1+3a+a^2)^t-1}||x_0-x^*||,\ t=1,2,\cdots, F=q_1^2 q_2\frac{\rho(x_0)^2}{\rho(y_0)\rho(z_0)},\nonumber\\
\end{eqnarray}
and $q_1$, $q_2$ are given by the first expression of the equation $(\ref{eqn:524})$.
\end{thm}
\begin{proof}
Let $x_0\in B(x^*, \delta)$ and ${x_t}$ be the sequence generated by three-step Newton-Traub method given in $(\ref{eqn:12})$. Let $\delta$, $q_1$ and $q_2$ be determined by the expressions $(\ref{eqn:522})$ and $(\ref{eqn:524})$, respectively. Assume that $x_t\in B(x^*, \delta)$. Then
\begin{eqnarray}\label{eqn:526}
||y_t-x^*||&=&||y_t-x^*-[G'(x_t)]^{-1}G(x_t)||\nonumber\\
&=&||[G'(x_t)]^{-1}[G'(x_t)(x_t-x^*)-G(x_t)+G(x^*)]||.\nonumber\\
\end{eqnarray}
Expanding $G(x_t)$  along $\ x^*$ from Taylor's expansion, we have
{\small
\begin{eqnarray}\label{eqn:527}
G(x^*)-G(x_t)+G'(x_t)(x_t-x^*)=G'(x^*)\int_{0}^{1}[G'(x^*)]^{-1}[G'(x_t)-G'(x_t^\tau)]d\tau(x_t-x^*).\nonumber\\
\end{eqnarray}}
Following the assumptions $(\ref{eqn:24})$ of the theorem and using the equations $(\ref{eqn:526})$ and $(\ref{eqn:527})$, it can be written as
{\small
\begin{eqnarray}\label{eqn:528}
&& ||y_t-x^*||\leq ||[G'(x_t)]^{-1}G'(x^*)||.||\int_{0}^{1}[G'(x^*)]^{-1}[G'(x_t)-G'(x^*)+G'(x^*)-G'(x_t^\tau)]d\tau||\nonumber\\
&&.||(x_t-x^*)||\nonumber\\
&\leq&\frac{1}{1-\int_{0}^{2\rho(x_t)}\kappa_0(u)du}\left(\int_{0}^{1}\int_{0}^{2\tau\rho(x_t)}\kappa_0(u)du\rho(x_t)d\tau+\int_{0}^{1}\int_{0}^{2\rho(x_t)}\kappa_0(u)du\rho(x_t)d\tau\right).\nonumber\\
\end{eqnarray}} 
In virtue of Lemma $(\ref{lm:21})$,  the above inequality takes the form
\begin{eqnarray*}\label{eqn:529}
||y_t-x^*||&\leq&\frac{2\int_{0}^{2\rho(x_t)}\kappa_0(u)du\rho(x_t)-\frac{1}{2}\int_{0}^{2\rho(x_t)}\kappa_0(u)udu}{1-\int_{0}^{2\rho(x_t)}\kappa_0(u)du}\nonumber\\
&\leq& \frac{2\int_{0}^{2\rho(x_t)}\kappa_0(u)du}{1-\int_{0}^{2\rho(x_t)}\kappa_0(u)du}\rho(x_t)= q_1\rho(x_t),
\end{eqnarray*}
which is equivalent to the first inequality of $(\ref{eqn:520a})$. Using a similar analogy and the second sub-step of the technique $(\ref{eqn:12})$, we may write
{\small
\begin{eqnarray}\label{eqn:5210}
&& ||z_{t}-x^*||\le||[G'(x_t)]^{-1}G'(x^*)||\left(||\int_{0}^{1}[G'(x^*)]^{-1}[G'(x_t)-G'(x^*)]d\tau||.||(y_t-x^*)||\right)\nonumber\\
&&+\left(||\int_{0}^{1}[G'(x^*)]^{-1}[G'(x^*)-G'(y_t^\tau)]d\tau||.||(y_t-x^*)||\right)\nonumber\\ 
&\leq&\frac{1}{1-\int_{0}^{2\rho(x_t)}\kappa_0(u)du}\left(\int_{0}^{1}\int_{0}^{2\tau\rho(y_t)}\kappa_0(u)du\rho(y_t)d\tau+\int_{0}^{1}\int_{0}^{2\rho(x_t)}\kappa_0(u)du\rho(y_t)d\tau\right). \nonumber\\
\end{eqnarray}}
As a result of Lemma $(\ref{lm:21})$, the above expression takes the form \\
{\small
\begin{eqnarray*}\label{eqn:5211}
||z_{t}-x^*||&\leq&\frac{\int_{0}^{2\rho(x_n)}\kappa_0(u)du\rho(y_t)+\int_{0}^{2\rho(y_t)}\kappa_0(u)du\rho(y_t)-\frac{1}{2}\int_{0}^{2\rho(y_t)}\kappa_0(u)udu}{1-\int_{0}^{2\rho(x_t)}\kappa_0(u)du}\nonumber\\
&\leq& \frac{\int_{0}^{2\rho(x_t)}\kappa_0(u)du\rho(y_t)+\int_{0}^{2\rho(y_t)}\kappa_0(u)du\rho(y_t)}{1-\int_{0}^{2\rho(x_t)}\kappa_0(u)du}\nonumber\\
&=& q_2q_1\rho(x_t),
\end{eqnarray*}}
Simultaneously, from the final sub-step of the scheme $(\ref{eqn:12})$, we can write
{\small
\begin{eqnarray}\label{eqn:5210a}
&&||x_{t+1}-x^*||\le||[G'(y_t)]^{-1}G'(x^*)||\left(||\int_{0}^{1}[G'(x^*)]^{-1}[G'(y_t)-G'(x^*)]d\tau||.||(z_t-x^*)||\right)\nonumber\\
&&+\left(||\int_{0}^{1}[G'(x^*)]^{-1}[G'(x^*)-G'(z_t^\tau)]d\tau||.||(z_t-x^*)||\right)\nonumber\\ 
&\leq&\frac{1}{1-\int_{0}^{2\rho(y_t)}\kappa_0(u)du}\left(\int_{0}^{1}\int_{0}^{2\tau\rho(y_t)}\kappa_0(u)du\rho(z_t)d\tau+\int_{0}^{1}\int_{0}^{2\rho(z_t)}\kappa_0(u)du\rho(z_t)d\tau\right).\nonumber\\
\end{eqnarray}}
Because of Lemma $(\ref{lm:21})$, the above expression leads to \\
{\small
\begin{eqnarray*}\label{eqn:5211a}
||x_{t+1}-x^*||&\leq&\frac{\int_{0}^{2\rho(y_t)}\kappa_0(u)du\rho(z_t)+\int_{0}^{2\rho(z_t)}\kappa_0(u)du\rho(z_t)-\frac{1}{2}\int_{0}^{2\rho(z_t)}\kappa_0(u)udu}{1-\int_{0}^{2\rho(y_t)}\kappa_0(u)du}\nonumber\\
&\leq& \frac{\int_{0}^{2\rho(y_t)}\kappa_0(u)du\rho(z_t)+\int_{0}^{2\rho(z_t)}\kappa_0(u)du\rho(z_t)}{1-\int_{0}^{2\rho(y_t)}\kappa_0(u)du}\nonumber\\
&=& q_3q_2q_1\rho(x_n),
\end{eqnarray*}}
where $q_1<1$, $q_2<1$ and $q_3<1$ are determined by the relation $(\ref{eqn:524})$.
Also, it can be seen that inequality $(\ref{eqn:523})$ may be easily derived from the second expression $(\ref{eqn:520a})$ and hence $x_t$ converges to $x^*$.\

 Furthermore, if the function $\kappa_a$ defined by the relation $(\ref{eqn:56})$ is non-decreasing for some $a$ with $0\le a\le1$ and $\delta$ is determined by the inequality $(\ref{eqn:522})$, it follows from the first inequality of the expression $(\ref{eqn:520a})$ and Lemma $(\ref{lm:22})$ that
\begin{eqnarray*}
||y_t-x^*||&\leq&\frac{2 \varphi_{0,a}(2\rho(x_t))2^a}{1-\int_{0}^{2\rho(x_t)} \kappa_0(u) du}\rho(x_t)^{a+1},\\
&\le&\frac{2\varphi_{0,a}(2\rho(x_0))2^a}{1-\int_{0}^{2\rho(x_0)} \kappa_0(u) du}\rho(x_t)^{a+1}
=\frac{q_1}{\rho(x_0)^a}\rho(x_t)^{a+1}.
\end{eqnarray*}
Moreover, from the second inequality of expression $(\ref{eqn:520a})$ and Lemma $(\ref{lm:22})$, we get
 \begin{eqnarray*}
||z_{t}-x^*||&\leq&\frac{ \varphi_{0,a}(2\rho(x_t))+ \varphi_{0,a}(2\rho(y_t)) (2\rho(x_t))^a \rho(y_t)}{1-\int_{0}^{2\rho(x_t)} \kappa_0(u) du}\rho(y_t),\\
&\leq&\frac{ \varphi_{0,a}(2\rho(x_0))+ \varphi_{0,a}(2\rho(y_0)) (2\rho(x_t))^a \rho(y_n)}{1-\int_{0}^{2\rho(x_t)} \kappa_0(u) du}\rho(y_t),\\
&=&\frac{q_1 q_2}{\rho(x_0)^a\rho(y_0)^a}\rho(x_t)^{2a+1}.
\end{eqnarray*}
Finally, the second inequality of expression $(\ref{eqn:520a})$ of the last sub-step and Lemma $(\ref{lm:22})$ gives
 \begin{eqnarray}\label{eqn:5212}
||x_{t+1}-x^*||&\leq&\frac{ \varphi_{0,a}(2\rho(y_t))+ \varphi_{0,a}(2\rho(z_t)) (2\rho(y_t))^a \rho(z_t)}{1-\int_{0}^{2\rho(y_t)} \kappa_0(u) du}\rho(z_t),\nonumber\\
&\leq&\frac{ \varphi_{0,a}(2\rho(y_0))+ \varphi_{0,a}(2\rho(z_0)) (2\rho(y_t))^a \rho(z_t)}{1-\int_{0}^{2\rho(y_t)} \kappa_0(u) du}\rho(z_t),\nonumber \\
&=&\frac{ q_1^{a+1} q_2 q_3}{\rho(x_0)^{a+a^2}\rho(y_0)^a\rho(z_0)^a}\rho(x_t)^{(a+1)^2+a}.
\end{eqnarray}
Now, we can prove that the expression $(\ref{eqn:525})$ is true for all integers $t\geq1$ through mathematical induction in the same line as previous result.
Accordingly, $x_t$ eventually converges to $x^*$.
\end{proof}


Next, the outcomes of Theorems $(\ref{th:51})$ and $(\ref{th:52})$ will then be re-captured using our newly improved theorems on some special functions $\kappa$.						
\begin{Corollary}\label{cr:53}
Suppose that $x^*$ satisfies $G(x^*)=0$, $G$ has a continuous derivative in  $B(x^*, \delta)$, $[G'(x^*)]^{-1}$ exists and $[G'(x^*)]^{-1}G'$ satisfies $(\ref{eqn:23})$, $(\ref{eqn:24})$ with $\kappa(u)=cau^{a-1}$ and  $\kappa_0(u)=c_0au^{a-1}$ i.e.:
\begin{equation}\label{eqn:531}
||[G'(x^*)]^{-1}(G'(x)-G'(y^\tau))||\leq c (1-\tau^a)(||x-x^*||+||y-x^*||)^a,
\end{equation}
and
\begin{eqnarray}\label{eqn:531a}
||[G'(x^*)]^{-1}(G'(x)-G'(x^*))||\leq c_02^a||x-x^*||^a,
\end{eqnarray}
$\forall \ x,\ y \in B(x^*, \delta),0\leq\tau\leq 1$, where $y^\tau=x^*+\tau(y-x^*)$, $\rho(x)=||x-x^*||$,$0<a<1$, $c>0$ and $c_0>0$. Let $\delta$ satisfy
\begin{eqnarray}\label{eqn:532}
\delta&=&\left(\frac{a+1}{2^a(c_0(a+1)+ca)}\right)^{\frac{1}{a}}.
\end{eqnarray}
Then, the three-step Newton-Traub method $(\ref{eqn:12})$ is convergent for all  \ $x_0\in B(x^*, \delta)$ and 
\begin{eqnarray*}\label{eqn:535}
||y_t-x^*||&\leq&\frac{ \int_{0}^{2\rho(x_t)} \kappa(u)u du}{2(1-\int_{0}^{2\rho(x_t)} \kappa_0(u) du)}\leq q_1\rho(x_t),
\end{eqnarray*}
\begin{eqnarray*}\label{eqn:536}
||z_{t}-x^*||&\leq&\frac{ \int_{0}^{\rho(x_t)+\rho(y_t)} \kappa(u)u du}{(\rho(x_t)+\rho(y_t))(1-\int_{0}^{2\rho(x_t)} \kappa_0(u) du)}\rho(y_t)\leq q_2 q_1\rho(x_t),\nonumber\\
\end{eqnarray*}
\begin{eqnarray*}\label{eqn:536a}
||x_{t+1}-x^*||&\leq&\frac{ \int_{0}^{\rho(y_t)+\rho(z_t)} \kappa(u)u du}{(\rho(y_t)+\rho(z_t))(1-\int_{0}^{2\rho(y_t)} \kappa_0(u) du)}\rho(z_t)\leq q_3 q_2 q_1\rho(x_t),\nonumber\\
\end{eqnarray*}
where the quantities
\begin{eqnarray}\label{eqn:534}
q_1=\frac{ca2^a\rho(x_0)^a}{(1+a)[1-2^ac_0\rho(x_0)^a]} ,\ q_2=\frac{ca(\rho(x_0)+\rho(y_0))^a}{(a+1)(1-2^ac_0\rho(x_0)^a)},\nonumber\\
q_3=\frac{ca(\rho(y_0)+\rho(z_0))^a}{(a+1)(1-2^ac_0\rho(y_0)^a)},
\end{eqnarray}
are less than $1$.
\end{Corollary}
\begin{Corollary}\label{cr:54}
Suppose that $x^*$ satisfies $G(x^*)=0$, $G$ has a continuous derivative in  $B(x^*, \delta)$, $[G'(x^*)]^{-1}$ exists and $[G'(x^*)]^{-1}G'$ satisfies $(\ref{eqn:24})$ with  $\kappa_0(u)=c_0au^{a-1}$ i.e.:
\begin{eqnarray}\label{eqn:541}
||[G'(x^*)]^{-1}(G'(x)-G'(x^*))||\leq c_02^a||x-x^*||^a, \forall \ x \in B(x^*, \delta),
\end{eqnarray}
where $\rho(x)=||x-x^*||$, $0<a<1$ and $c_0>0$. Let $\delta$ satisfy
\begin{eqnarray}\label{eqn:542}
\delta&=&\left(\frac{1}{3c_02^a}\right)^{\frac{1}{a}}.
\end{eqnarray}
Then, the three-step Newton-Traub method $(\ref{eqn:12})$ is convergent for all  \ $x_0\in B(x^*, \delta)$ and 
\begin{eqnarray*}\label{eqn:545}
||y_t-x^*||&\leq&\frac{2 \int_{0}^{2\rho(x_t)} \kappa_0(u) du}{1-\int_{0}^{2\rho(x_t)} \kappa_0(u) du}\rho(x_t)\leq q_1\rho(x_t), \nonumber\\
||z_{t}-x^*||&\leq&\frac{ \int_{0}^{2\rho(x_t)} \kappa_0(u) du+ \int_{0}^{2\rho(y_t)} \kappa_0(u) du}{1-\int_{0}^{2\rho(x_t)} \kappa_0(u) du}\rho(y_t)\leq q_2 q_1\rho(x_t),\nonumber\\
||x_{t+1}-x^*||&\leq&\frac{ \int_{0}^{2\rho(y_t)} \kappa_0(u) du+ \int_{0}^{2\rho(z_t)} \kappa_0(u) du}{1-\int_{0}^{2\rho(y_t)} \kappa_0(u) du}\rho(z_t)\leq q_3 q_2 q_1\rho(x_t),\nonumber\\
\end{eqnarray*}
where the quantities
\begin{eqnarray}\label{eqn:544}
&& q_1=\frac{c_02^{a+1}\rho(x_0)^a}{1-2^ac_0\rho(x_0)^a} ,\ q_2=\frac{c_02^a(\rho(x_0)^a+\rho(y_0)^a)}{1-2^ac_0\rho(x_0)^a},\nonumber\\
&& q_3=\frac{c_02^a(\rho(y_0)^a+\rho(z_0)^a)}{1-2^ac_0\rho(y_0)^a},
\end{eqnarray}
are less than 1.
\end{Corollary}

\begin{Corollary}\label{cr:55}
Suppose that $x^*$ satisfies $G(x^*)=0$, $G$ has a continuous derivative in  $B(x^*, \delta)$, $[G'(x^*)]^{-1}$ exists and $[G'(x^*)]^{-1}G'$ satisfies $(\ref{eqn:24})$ with $\kappa_0(u)=\frac{2\gamma c_0}{(1-\gamma u)^3}$ i.e.:
\begin{eqnarray}\label{eqn:551}
||[G'(x^*)]^{-1}(G'(x)-G'(x^*))||\leq  \frac{c_0}{(1-2\gamma\rho(x))^2}-c_0, \forall \ x \in B(x^*, \delta), \nonumber\\
\end{eqnarray}
where $\rho(x)=||x-x^*||$, $\gamma>0$ and $c_0>0$. Let $\delta$ satisfy
\begin{eqnarray}\label{eqn:552}
\delta&=&\frac{3c_0+1-\sqrt{3c_0(3c_0+1)}}{2\gamma(3c_0+1)}.
\end{eqnarray}
Then, three-step Newton-Traub method $(\ref{eqn:12})$ is convergent for all  \ $x_0\in B(x^*, \delta)$ and 
\begin{eqnarray*}\label{eqn:555}
||y_t-x^*||&\leq&\frac{2 \int_{0}^{2\rho(x_t)} \kappa_0(u) du}{1-\int_{0}^{2\rho(x_t)} \kappa_0(u) du}\rho(x_t)\leq q_1\rho(x_t), \nonumber\\
||z_{t}-x^*||&\leq&\frac{ \int_{0}^{2\rho(x_t)} \kappa_0(u) du+ \int_{0}^{2\rho(y_t)} \kappa_0(u) du}{1-\int_{0}^{2\rho(x_t)} \kappa_0(u) du}\rho(y_t)\leq q_2 q_1\rho(x_t),\nonumber\\
||x_{t+1}-x^*||&\leq&\frac{ \int_{0}^{2\rho(y_t)} \kappa_0(u) du+ \int_{0}^{2\rho(z_t)} \kappa_0(u) du}{1-\int_{0}^{2\rho(y_t)} \kappa_0(u) du}\rho(z_t)\leq q_3 q_2 q_1\rho(x_t),\nonumber\\
\end{eqnarray*}
where the quantities
{\small
\begin{eqnarray}\label{eqn:554}
q_1&=&\frac{2c_0-2c_0(1-2\gamma\rho(x_0))^2}{[1-2\gamma \rho(x_0)]^2(1+c_0)-c_0},\nonumber\\
q_2&=&\frac{[c_0-c_0(1-2\gamma\rho(x_0))^2](1-2\gamma\rho(y_0))^2)+[c_0-c_0(1-2\gamma\rho(y_0))^2](1-2\gamma\rho(x_0))^2)}{([1-2\gamma \rho(x_0)]^2(1+c_0)-c_0)(1-2\gamma\rho(y_0))^2)},\nonumber\\
q_3&=&\frac{[c_0-c_0(1-2\gamma\rho(y_0))^2](1-2\gamma\rho(z_0))^2)+[c_0-c_0(1-2\gamma\rho(z_0))^2](1-2\gamma\rho(y_0))^2)}{([1-2\gamma \rho(y_0)]^2(1+c_0)-c_0)(1-2\gamma\rho(z_0))^2)},\nonumber\\
\end{eqnarray}}
are less than 1.
\end{Corollary}

\textbf{Remark 5.4} $(a)$ If $\kappa_0=\kappa$, then our results specialize to earlier ones [\cite{Chen}, \cite{Li}, {\cite{Argyros}}, {\cite{Wang}}, {\cite{Wang2}}]. But if $\kappa_0<\kappa$, then the benefits stated in the abstract  are obtained (see also Example $6.1$ and Example $6.2$).\\
$(b)$ A further extension can be achieved as follows. Suppose $(\ref{eqn:22})$ holds and equation $2\kappa_0(u)u-1=0$ has a minimal positive zero $\overline{\delta}$. Define\\
\~{B}$= B(x^*, \delta) \cap B(x^*, \overline{\delta})$. Moreover, suppose
\begin{eqnarray}\label{eqn:556}
||G(x)-G(y^\tau)||\leq \int_{\tau(\rho(x)+\rho(y))}^{\rho(x)+\rho(y)}\overline{\kappa}(u) du, 
\end{eqnarray}
where $\forall \ x,\ y \in$\~{B}, $0\leq\tau \leq1,\ and\ \overline{\kappa}$ is as $\kappa$. Then, we have
\begin{eqnarray*}\label{eqn:557}
\overline{\kappa}(u)&\leq& \kappa(u)\  for\ all\  u\in[0, min\{\delta, \overline{\delta}\}]. 
\end{eqnarray*}
Then, in view of the proofs $\overline{\kappa}$ can replace $\kappa$ in all results with $\kappa$. But, then if 
 \begin{eqnarray*}\label{eqn:558}
\overline{\kappa}(u)< \kappa(u),
\end{eqnarray*}
the benefits stated in the introduction are extended even further. In the case of the motivational example, we have
\begin{eqnarray*}\label{eqn:559}
\kappa_0<\overline{\kappa}=\frac{e^{\frac{1}{(e-1)}}}{2}<\kappa.
\end{eqnarray*}




\section{\bf Numerical Examples}

\begin{Example} 
Returning to the motivational example presented in the study's introduction, using $(\ref{eqn:32})$ and $G'(x^*)=(1,1,1)^T$, we have:\\
Old case $\kappa_0(u)=\kappa(u)=\frac{e}{2}$ gives
\begin{eqnarray*}
\delta_0=0.245253.
\end{eqnarray*}
Case $\kappa_0(u)= \frac{e-1}{2}\ and\ \kappa(u)=\frac{e}{2}$ gives
\begin{eqnarray*}
\delta_1=0.324947.
\end{eqnarray*}
Case $\kappa_0(u)= \frac{e-1}{2} \ and\ \overline{\kappa}(u)=\frac{e^{\frac{1}{(e-1)}}}{2}$ gives
\begin{eqnarray*}
\delta_2=0.382692.
\end{eqnarray*}
Notice that
\begin{eqnarray*}
\delta_0<\delta_1<\delta_2.
\end{eqnarray*}
\end{Example}
\begin{Example} 
Choose $X=Y=C[0,1], \Omega=\overline{V}(0,1)$ and $x^*=0$. Then, define $t$ on $\Omega$ as
 \begin{eqnarray*}
t(h)(x)=h(x)-\int_{0}^{1}x\tau h(\tau)^3 d\tau.
\end{eqnarray*}
So,
\begin{equation*}
t'(h(p))(x)= p(x)-3\int_{0}^{1}x\tau h(\tau)^2p(\tau) d\tau\ for\ all\ p\in \Omega.
\end{equation*}
Then, we get
\begin{eqnarray*}
\kappa_0(u)= 1.5u < \kappa(u)= \overline{\kappa}(u)= 3u.
\end{eqnarray*}
As a result, we get the same advantages as in Example $6.1$ by solving $(3.1)$.
\end{Example}
\begin{Example} 

Let $X=Y=R$, the reals. Define
\begin{eqnarray*}\label{eqn:ex1}
G(x)= \int_{0}^{x}\left(1+2x\sin\frac{\pi}{x}\right)dx,\ \forall x \in R.  
\end{eqnarray*}
Then
\[
  G'(x) =
  \begin{cases}
                                   1+2x\sin\frac{\pi}{x}, & x\neq0, \\
                                   1, & x=0, 
  \end{cases}
\]
Obviously, $x^*=0$ is a zero of $G$ and $G'$ satisfies that 
\begin{eqnarray*}\label{eqn:ex2}
||[G'(x^*)]^{-1}(G'(x)-G'(x^*))||= \left|2x\sin\frac{\pi}{x}\right|\leq2|x-x^*|, \forall \ x \in R.
\end{eqnarray*}
It follows from Theorem $(\ref{th:52})$ that for any $x_0\in B(x^*,1/6)$
\begin{eqnarray*}\label{eqn:ex3}
||x_t-x^*||\leq F^{5^t-1}||x_0-x^*||,\ t=1,2,\cdots, F=\left(\frac{16|x_0|^4[2|x_0|+2|y_0|]}{[1-2|x_0|]^3|y_0||z_0|}\right).
\end{eqnarray*}
However, there is no positive integrable function $L$ such that the\\ inequality $(\ref{eqn:23})$ is satisfied. In fact, notice that 
 \begin{equation*}\label{eqn:ex4}
||[G'(x^*)]^{-1}(G'(x)-G'(y^\tau))||=\left|2x\sin\frac{\pi}{x}-2y\tau\sin\frac{\pi}{y\tau}\right|=\frac{4}{2k+1},
\end{equation*}
for $x=1/k, y=1/k, \tau=\frac{2k}{2k+1}$ and $k=1,2,\cdots$ Thus, if there was a positive integrable function $L$ such that the inequality $(\ref{eqn:23})$ holds on $V(x^*, r)$ for some $r>0$, it follows that there exists some $n_0>1$ such that
\begin{eqnarray*}\label{eqn:ex5}
\int_{0}^{2r\delta}\kappa(u)du\geq \sum_{k=n_0}^{+\infty}\int_{\frac{4}{2k+1}}^{\frac{2}{k}}L(u)du\geq \sum_{k=n_0}^{+\infty}\frac{4}{2k+1}=+\infty,
\end{eqnarray*}
which is a contradiction. This example indicates that Theorem $(\ref{th:52})$ is a significant improvement over Theorem $(\ref{th:51})$ if the radius of the convergence ball is neglected.
\end{Example}

\begin{center}
{\bf Conclusion}
\end{center} 
A novel technique is developed in order to provide a finer local convergence analysis without making additional assumptions than in earlier studies. The method is quite generic. It turns out that, while the criteria are more generic, they are also more flexible, which results in some benefits with no more computational cost. Hence, we have extended the applicability of modified Newton’s method in cases not covered before. Our approach paves the way for future research to improve local results for Newton-type methods and other iterative procedures.


\end{document}